\documentclass[final, 3p, times, authoryear]{elsarticle}

\usepackage{amssymb}

\usepackage{amsthm}

 \usepackage{lineno}

\usepackage{setspace}

\usepackage{float}

\biboptions{sort}

\usepackage{amsmath, amsthm, amssymb} 
\usepackage{bm} 

\usepackage{epsfig} 
\usepackage{graphicx} 

\def\bea{\begin{eqnarray}}
\def\eea{\end{eqnarray}}
\def\bean{\begin{eqnarray*}}
\def\eean{\end{eqnarray*}}
\renewcommand\eqref[1]{(\ref{#1})}

\def\R{\mathbb{R}}
\def\N{\mathbb{N}}

\def\Ind#1{\,\mathbb{I}\{#1\}\,}

\theoremstyle{plain}
\newtheorem{theorem}{Theorem}

\newtheorem{corollary}[theorem]{Corollary}

\theoremstyle{definition}

\theoremstyle{remark}

\theoremstyle{note}
\newtheorem{note}[theorem]{Note}
\usepackage{comment}
\usepackage{ulem} 
\usepackage[usenames,dvipsnames]{xcolor}
\definecolor{orange}{rgb}{1,0.5,0}
\definecolor{green}{rgb}{0.513,0.73,0.442}
\renewcommand\emph[1]{{\it #1}} 
\ifx \q \undefined
  \def \q[#1][#2][#3]{q_{#2}^{#3}\left(#1\right)}

\def\R{\mathbb{R}} 
\def\N{\mathbb{N}} 

\def\dt{{h}} 

\def\Ind#1{\,\mathbb{I}\{#1\}\,}  

\excludecomment{repeatIntro}
\excludecomment{styleRevisionFormat}
\excludecomment{styleRevisionColor}
\definecolor{NeonYellowWhiteOnBlack}{rgb}{0.016,0.009, 0.975}
\def\kwy#1{#1}
\begin{styleRevisionColor}
\def\kwy#1{\textcolor{NeonYellowWhiteOnBlack}{#1}}
\end{styleRevisionColor}

\definecolor{NeonBlueWhiteOnBlack}{rgb}{0.975, 0.016, 0.009}
\def\kwb#1{#1}
\begin{styleRevisionColor}
\def\kwb#1{\textcolor{NeonBlueWhiteOnBlack}{#1}}
\end{styleRevisionColor}

\definecolor{NeonOrangeWhiteOnBlack}{rgb}{0.0075, 0.44, 0.5525}
\def\kwo#1{#1}
\begin{styleRevisionColor}
\def\kwo#1{\textcolor{NeonOrangeWhiteOnBlack}{#1}}
\end{styleRevisionColor}

\def\kwdb#1{#1}
\begin{styleRevisionColor}
\def\kwdb#1{\textcolor{orange}{#1}}
\end{styleRevisionColor}

\def\kwp#1{#1}
\begin{styleRevisionColor}
\def\kwp#1{\textcolor{OliveGreen}{#1}}
\end{styleRevisionColor}

\definecolor{reproduction-gray}{gray}{0.65}

\def\paragraphTopicSentence#1{#1}
\begin{styleRevisionFormat}
\def\paragraphTopicSentence#1{
	\color{summarizing-gray}
	\emph{#1}
	\color{black}
}
\end{styleRevisionFormat}
\definecolor{summarizing-gray}{gray}{0.35}

\usepackage[geometry]{ifsym}

\newcounter{sectionSentenceCounter}[section] 
\newcounter{subsectionSentenceCounter}[subsection] 
\begin{document}

\begin{frontmatter}
%
%
%
%
\title{Trajectory composition of Poisson time changes and Markov counting systems}
¥
\author{Carles Bret\'{o}} \fnref{label1}
\fntext[label1]{Tel. +34916245855; Fax:+34916249848}
\ead{carles.breto@uc3m.es}
\address{Departamento de Estad\'{i}stica and Instituto Flores de Lemus, Universidad Carlos III de Madrid, C/ Madrid 126, Getafe, 28903, Madrid, Spain}
%
%
\begin{abstract}
Changing time of simple continuous-time Markov counting processes by independent unit-rate Poisson processes results in Markov counting processes for which we provide closed-form transition rates via composition of trajectories and with which we construct novel, simpler infinitesimally over-dispersed processes.
\end{abstract}
%
%
\begin{keyword}
Compartmental models\sep
Subordination\sep
Random time change\sep
Compound process\sep
Infinitesimal dispersion
\end{keyword}
\end{frontmatter}

%
%
¥
\section{Introduction}
\label{sec:introduction}
%
%
The \kwp{statistical analysis of dynamical systems} plays an important role in scientific research. 
When these \kwp{systems} involve counts, such analysis may be carried out using continuous-time Markov \kwdb{processes}, \kwdb{which} are often approximations to real systems and hence fail to capture some \kwp{features} of real data. 
Luckily, some of these \kwp{features} may be better captured after replacing time in those processes by a Markov \kwdb{random time}. 
Such a \kwdb{time randomization} approach to improving statistical modeling was recently proposed and studied in detail in \cite{breto2011}, where the resulting Markov time-changed processes are defined via \kwp{transition rates}. 
%
Unfortunately, \kwp{transition rates} of such time-changed processes are in general unavailable and can be difficult to obtain in \kwdb{closed form}. 
%
This lack of \kwdb{closed-form} transition rates limits the appeal of this time randomization approach and may even discourage applied researchers from using \kwp{it} at all. 
%
%
To help make this \kwp{approach} more appealing, this paper considers changing by a Poisson process the time of a large family that includes many continuous-time Markov counting processes used in applications (for example from epidemiology, biochemistry or sociology). 
For the resulting time-changed process, transition rates are provided in the required \kwdb{closed form}. 
These \kwdb{closed-form transition rates} constitute the \kwp{main result} of this paper, \kwp{which} we obtain by composing trajectories of the counting process with those of the random time (instead of by integrating out the random \kwdb{time}). 
Our choice of a \kwdb{Poisson time} change seems to be unusual in the applied literature and produces time-changed models simpler than those previously considered, as we illustrate by constructing several novel over-dispersed counting processes, which can be used as building blocks to construct multivariate over-dispersed Markov counting systems. 

%
%
%
¥
\paragraphTopicSentence{Dynamical systems that involve counts have been studied in many \kwy{disciplines} by considering \kwo{Markov counting systems} without simultaneous events, although \kwp{compound systems} (which allow simultaneity) have also received some attention.} 
%
\kwy{Fields} where counting systems have been modeled as continuous-time Markov chains include epidemiology and ecology \citep{kermack1927, shrestha2011}, pharmacokinetics \citep{matis1979, haseltine2002, srivastava2002} and engineering and operations research \citep{doig1957, jackson2002}. 
In these fields, most such processes are in fact \kwo{Markov counting systems} \citep{breto2011}, mainly networks of queues \citep{bremaud1999} or compartmental models \citep{jacquez1996, matis2000} that rule out the possibility of simultaneous transitions or events. 
When simultaneous events are possible, these \kwo{counting systems} are called compound \citep{breto2011}. 

%
\paragraphTopicSentence{\kwp{Compound counting systems} can capture better the variability in real data thanks to being \kwy{infinitesimally over-dispersed} and have been constructed relying on \kwo{random time changes} and defined via \kwb{closed-form transition rates}, which is what this \kwdb{paper is concerned} with.}
%
\kwp{Compound Markov counting systems} have been considered as a means to increase compatibility of theoretical models with real data, for example in the context of DNA sequence alignment and genomic data \citep{thorne1992} and of environmental stochasticity and epidemiological data \citep{breto2009}. 
\kwp{They} are also \kwy{infinitesimally over-dispersed} \citep{breto2012}, which is a model feature favored by infectious-disease data \citep{breto2009, he2009, ionides2006-pnas, shrestha2011}. 
Such \kwy{over-dispersion} can be modeled with compound processes, which can be constructed under mild conditions \citep{breto2012b} via the well-known operation of random change of time. 
Such \kwo{time randomization} approach was considered in detail in \citet{breto2011}, after being first illustrated in \cite{breto2009} where a compound compartmental model was constructed and defined by transition rates expressed in closed form. 
Investigating such closed-form rates for other models in general is \kwdb{our main concern}. 

%
\paragraphTopicSentence{The \kwdb{problem that this paper} addresses is the difficulty deriving closed-form transition rates of time-changed processes, which lies in the non-linearity of \kwo{expected values} of \kwy{transition probabilities} and which \kwp{limits the appeal} of time randomization in applications.}
%
Transition rates of Markov counting processes can be understood as appropriate limits of transition probabilities \citep{bremaud1999}. 
These \kwy{probabilities} are most likely non-linear in time. 
After time is randomized, transition rates are instead determined by the expected transition probabilities (with respect to the randomized time), but such \kwo{expected values} need not be readily available in closed form due to the non-linearity. 
Consider a unit-rate Poisson process whose time index $t$ is changed by random time $R(t)$. 
Its \kwo{expected} probability of $k$ transitions over time interval $[0,l\,]$ is the left hand side of~\eqref{eqn:gamma-rate} below, which is an analytic expression. 
A corresponding closed-form expression can be obtained, for example, assuming that $\{R(t)\}$ is a gamma process with 
$E \big[ R(t) \big] = t$ and $V \big [ R(t) \big ] = t / \tau$. 
Such expression is \citep{breto2011, kozubowski2009}, if $\Gamma$ is the gamma function, 
\bea
\label{eqn:gamma-rate}
E_R \Big[ R(l)^k e^{-R(l)} \Big]/k! &=& \Gamma \bigl( l/\tau + k \bigr) / \Bigl( k!\, \Gamma \bigl( l/\tau \bigr) \bigl(1 + \tau \bigr) ^{ l/\tau} \bigl( 1 + \tau^{-1} \bigr) ^{k} \Bigr).
\eea 
However, this closed-form expression is based on derivations specific to the Poisson gamma process of this example and it need not extend straightforwardly to other random times or counting processes (like the non-linear death processes considered in Section~\ref{sec:applications}). 
Seemingly technical difficulties like this one can \kwp{prevent} applied researchers from randomizing time. 

%
\paragraphTopicSentence{The \kwp{discouraging} limitations imposed on time randomization by unavailable closed-form expressions stem from the necessity to define time-changed models as \kwy{implicit} hierarchies and the resulting complications on \kwo{model interpretation}, which we seek to \kwdb{alleviate in this paper}.}
%
\kwy{Implicit definitions} of a model are those given in terms of numerical procedures to generate realizations or sample paths \citep{breto2009}, 
\kwy{e.g.}, for our Poisson gamma example above, a realization at time $t$ would come from the following hierarchy of random draws: use a value drawn from a gamma random variable with mean $t$ (and variance $t/\tau$) as the mean of a Poisson random variable from which to draw the desired process realization. 
\kwy{Implicit definitions} like this one are all that is needed to do inference using ``plug-and-play'' methods \citep{breto2009}, without the need to work out any closed-form expressions like~\eqref{eqn:gamma-rate}. 
However, \kwy{implicit models} can be harder to interpret. 
Consider what an applied researcher might ask when deciding what to make of and how to interpret results obtained from an implicit model: 
Is the time-changed model well-behaved? 
What aspects of the original model vary after changing time? 
Should interpretation of the original parameters change? 
How does the choice of time change affect the answers to these questions?
\kwo{Interpretation issues} like these might be tackled considering implicit definitions only but answers may reflect numerical artifacts and may not be as apparent as with closed-form expressions, as illustrated in next section. 
Helping mitigate such \kwo{interpretation issues} to make time randomization more attractive is the ultimate \kwdb{goal of this paper}. 

%
\paragraphTopicSentence{The \kwdb{main contribution} of this paper is to provide \kwy{closed-form transition rates} for a large family of Markov counting processes time-changed by \kwo{Poisson times} via composition of trajectories and to illustrate how these closed-form rates \kwb{facilitate the use of time randomization} to improve Markov counting systems used in applications.} 
%
The sought \kwy{closed-form expressions} are provided in 
Section~\ref{sec:thm} 
under mild requirements satisfied by many well-behaved processes considered in the applied literature. 
These \kwy{expressions} provide the desired details about the time-changed process to help address interpretation issues, promoting the use of time randomization. 
\kwy{They} are obtained by focusing on process trajectories to get around non-linear expectations like~\eqref{eqn:gamma-rate}, for which the unusual Poisson time change turns out to be convenient. 
Not only does our \kwo{Poisson time choice} facilitate the derivation of the expressions, it also avoids increasing the number of parameters, which results in a simpler interpretation of the parameters of the time-changed models. 
This is illustrated in Section~\ref{sec:applications}, where \kwb{time randomization} is considered for several processes of interest in the biological and social sciences and applied in detail to the widespread Poisson, linear birth and linear death processes to construct novel, infinitesimally over-dispersed processes and multivariate systems without the interpretation issues of time-changed models defined only implicitly. 
%
%
¥
\section{Closed-form transition rates of simple Markov counting processes time-changed by Poisson processes}
\label{sec:thm}
Before giving the closed-form time-changed rates in Theorem~\ref{thm:rates} below, 
we introduce our notation as we progress through some fundamental aspects of \kwdb{counting processes}, \kwp{random change of time} and \kwdb{Poisson processes}, which provides a context for the \kwp{theorem}.

\kwdb{Continuous-time Markov counting processes} are fully characterized by \kwy{transition rates}, which are intimately related to the process \kwo{compoundness} and \kwb{intensity}, and which will also be used to define our \kwp{time-changed processes}. 
We denote a time-homogeneous continuous-time Markov counting process by $\{X(t)\}$ and define it via its transition rates or transition semigroup local characteristics \citep{bremaud1999}, which we write as
\begin{eqnarray}
\label{eqn:transition-rates}
\q[x,k][X][] &\equiv& \lim\limits_{\dt \downarrow 0} \frac{P \Big( X(t+\dt) = x + k \; | \; X(t)=x \Big) }{\dt}
\end{eqnarray}
where $\dt, t \in \R_{\geq 0}$, $x \in \N_{\geq 0}$ and $k \in \N_{\geq 1}$. 
\kwy{Transition rates} determine whether $\{X(t)\}$ is infinitesimally over-dispersed, which occurs if and only if simultaneous events are possible \citep{breto2011}---i.e., if and only if $\{X(t)\}$ is compound so that there exists at least one $x$  and $k^{\star} > 1$ with $\q[x,k^{\star}][X][] > 0$. 
If $\{X(t)\}$ is not \kwo{compound}, 
then it is {\it simple} \citep{daley2003} and $\q[x,1][X][]$ is the only non-zero transition rate.
\kwy{Transition rates} are also related to the process rate function or intensity \citep{daley2003}, which we write as
\bea
\nonumber 
\lambda_{X}(x) \equiv \lim\limits_{\dt \downarrow 0}\frac{1 - P \Big( X(t + \dt) = x\; | \; X(t)=x \Big)}{\dt}. 
\eea 
If the \kwb{rate function} summarizes the overall activity implied by the transition rates and is finite---i.e., if it satisfies $\lambda_X(x) = \sum_{k\ge 1}\q[x,k][X][] < \infty$ for all $x$, we say that $\{X(t)\}$ is conservative and stable \citep{bremaud1999}. 
Conservation, stability and simpleness are satisfied by many well-behaved processes used in applications (including those in Section~\ref{sec:applications} below) and are the only conditions that Theorem~\ref{thm:rates} imposes on the counting process to be \kwp{time-changed}. 
(Actually, time homogeneity is also imposed but it can be relaxed at the cost of complicating notation and derivations.)

\kwp{Random change of time} only needs to be \kwy{sketched} at this point and details are postponed to~\ref{app:proof}, but we do elaborate here on our reasons for \kwdb{choosing a Poisson} time change over other time changes. 
\kwy{In a nutshell}, random time change involves a base process, say $\{X(t)\}$, whose time index $t$ is replaced by a random time process, say $\{N(t)\}$, giving the time-changed process, say $\{S(t)\} \equiv \{X( N(t) ) \}$. 
While the base process considered in this paper is rather general, i.e., simple conservative stable processes, our time change is very specific: a unit-rate \kwdb{Poisson process}. 

The choice of \kwdb{Poisson} time change is \kwy{unusual in the literature} but key if we wish to avoid directly integrating out the random time from the \kwo{hierarchy} $X \big( N(t) \big)$ by following our alternative approach based on \kwb{trajectory composition} to prove \kwp{Theorem~\ref{thm:rates}} below. 
In the time randomization \kwy{literature}, L\'{e}vy processes \citep{sato1999} are routinely considered as time changes. 
One such process that has recurrently been chosen to obtain closed-form expressions is the gamma process. 
Examples include time-changed diffusions \citep{madan1998} or, in the context of counting processes, time-changed Poisson processes \citep{hougaard1997, kozubowski2009, lee1993} and time-changed pure-birth and pure-death processes \citep{breto2011}. 
In spite of processes other than gamma having been considered 
\citep{marion2000, varughese2008, 
kumar2011}, Poisson processes seem to have been dispensed with. 
\kwy{This literature} emphasizes the random-variable hierarchy perspective, which drags us to the problematic non-linear expected transition probabilities analogous to~\eqref{eqn:gamma-rate}. 
This \kwo{hierarchical perspective} gives the following analytical expression for the time-changed transition rates using~\eqref{eqn:transition-rates} directly 
\begin{eqnarray}
\label{eqn:analytic}
\lim\limits_{\dt \downarrow 0} \frac{E_{N (t+\dt)} \Bigg[ \; P \Big( \; B \big(N (t+\dt) \big) = s + k \; \; \big| \; B \big(N (t) \big) = s \; \Big) \; \Bigg]}{\dt}.
\end{eqnarray}
This expression is not in closed form. 
To obtain one, the limit could be handled using Taylor expansions but the expectation is likely to be non-linear (as argued in the introduction). 
Such non-linear expectations can be circumvented by composing the trajectories of the processes (more on this in \ref{app:proof}). 
\kwb{Trajectories are easier to compose}, as we argue in \ref{app:proof}, if the time change is 
a Poisson process like in \kwp{Theorem~\ref{thm:rates}}.

%
%
\begin{theorem}[\textbf{Transition rates of simple Markov counting processes time-changed by Poisson processes}]
\label{thm:rates}
Let $\{X(t)\}$ be a continuous-time Markov counting process that is time-homogeneous, simple, conservative, stable, and defined by rate function $\lambda_{X}(x) = \q[x,1][X][] < \infty$. 
Let $\{N(t)\}$ be a Poisson process defined by $\lambda_N(n)=1$.
Let also both processes be 
independent.
Then, the transition rates and rate function of the time-changed process $\{S(t)\} \equiv \Big\{ X\big(N(t)\big) \Big\}$ are
\bea
\label{eqn:time-changed-rates}
\q[s,k][S][] &=& P \Big (X(1)=s+k \; \big | \;  X(0) = s \Big )\\
\label{eqn:time-changed-rate}
\lambda_{S}(s) &=& 1-e^{-\lambda_X (s)}.
\eea
\end{theorem}
\begin{proof}
Theorem~\ref{thm:rates} is proved in \ref{app:proof} to preserve the flow of the paper.
\end{proof} 
%
%
¥
\section{Applications: making time change more appealing to model over-dispersed systems}
\label{sec:applications}
The closed-form expressions in \kwp{Theorem~\ref{thm:rates}} make time randomization more attractive to applied researchers by providing them with details about their univariate time-changed process $\{S(t)\}$ that refer to the \kwdb{overall process behavior} and to the \kwp{role of specific parameters}, as we illustrate below with \kwdb{different time changes} and \kwdb{base processes} (even when \kwp{the distribution of $X(1)$ in~\eqref{eqn:time-changed-rates} is difficult to obtain}). 
This \kwp{distribution} only needs to be considered for univariate processes, which can then be used as building blocks to construct infinitesimally over-dispersed \kwdb{multivariate systems of counts}. 

The \kwdb{overall behavior of a process} is a subjective concept, which could include features like whether or not it is \kwy{well-behaved}, \kwo{compound} or whether its \kwb{transition rates and rate function} have any distinctive characteristics, all of which can be complemented by \kwp{additional information on the role of specific parameters}. 
The Poisson time-changed process $\{S(t)\}$ of Theorem~\ref{thm:rates} is \kwy{well-behaved} in the sense that its transition rates $\q[s,k][S][] \in [0,1] < \infty$ and that it remains conservative and stable, since $\lambda_{S}(s) = 1 - P \big (X(1)=s \; | \;  X(0) = s \big ) = \sum_{k\ge 1}\q[s,k][S][]$. 
\kwy{This} assures modelers that their time-changed model is not pathological and that it fits in the framework of most existing Markov process theory. 
$\{S(t)\}$ is also \kwo{compound}, which implies it is infinitesimally over-dispersed. 
Such \kwo{over-dispersion} informs applied scientists of the additional variability demanded by the data, which may also be interpreted as the base model $\{X(t)\}$ being misspecified \citep{breto2009, breto2011}. 
\kwo{Compoundness} also has implications for describing and summarizing $\{S(t)\}$. 
\kwo{It} suggests complementing the information given by the more conventional infinitesimal mean and variance with that given by the rate function and the mean and variance of the transition rates (or jump sizes), which have the following distinctive features in the case of $\{S(t)\}$. 
One on hand, moments of the \kwb{transition rates} in~\eqref{eqn:time-changed-rates} are given directly by the distribution of random variable $X(1)$. 
On the other hand, the \kwb{rate function} in~\eqref{eqn:time-changed-rate} is necessarily restricted to the interval $[0,1]$. 
\kwb{These} two features clue in modelers that event size moments will inherit any constraint present in the distribution of $X(1)$ and that the event rate function of $\{S(t)\}$ is bounded from above. 
Moreover, the closed form for the \kwb{rate function} given in \eqref{eqn:time-changed-rate} permits quantitatively assessing the impact on $\lambda_S(s)$ of \kwp{parameters} appearing in the original $\lambda_X(x)$. 

The \kwp{role of specific parameters} can be analyzed in more detail by using additional information on the \kwy{distribution} of $X(1)$ given $X(0)$ to measure the \kwo{impact of a parameter on transition rates, on the rate function and on moments}, which in addition facilitates \kwb{connecting and comparing} the time-changed process to other processes, as we illustrate in \kwdb{several examples} below. 
The required \kwy{distribution} of $X(1)$ is a ubiquitous result for Poisson base processes $\{X(t)\}$ with rate $\alpha$ (independent of Poisson time change $\{N(t)\}$), for which $X(1) \sim \text{Poisson}(\alpha)$. 
This \kwy{distribution} is all that is needed to apply Theorem~\ref{thm:rates} (from which we borrow the notation) to prove the following corollary. 
\begin{corollary}[\textbf{Poisson Poisson process}]\label{cor:PP}
Consider Theorem~\ref{thm:rates} and let $ \lambda_X(x) = \alpha$, i.e., let $\{X(t)\}$ be a Poisson process with rate $\alpha \in \R_{>0}$. 
The transition rates of $\{S(t)\}$ are the probability mass function of a Poisson random variable parameterized by rate $\alpha$, i.e., $\q[s,k][S][] = \alpha^k e^{-\alpha} / k!$ and its rate function is $\lambda_{S}(s) = 1-e^{-\alpha}$.
\end{corollary}
\begin{note}[\textbf{On Poisson Poisson processes}]
The name ``Poisson Poisson process'' has been used in the literature to refer to a different process, i.e., to refer to compound Poisson processes with Poisson event size distribution \citep[see for example][]{hanson2007}. 
\end{note}
Corollary~\ref{cor:PP} gives a rate function and transition rates for $\{S(t)\}$ parameterized solely by $\alpha$. 
The \kwo{impact of $\alpha$ on transition rates, the rate function and moments} for $\{S(t)\}$ is different than for $\{X(t)\}$. 
Before the time change, an increase in $\alpha$ linearly increases the rate $\lambda_X$ at which events occur in $\{X(t)\}$ and events have invariably size one. 
After the time change, its \kwo{impact} on the event rate $\lambda_S$ becomes non-linear and, more importantly, it changes the mean and variance of event sizes of $\{S(t)\}$. 
These event size moments are simply those of $X(1)$, i.e., both $\alpha$ and, although knowing them is informative, they differ from the infinitesimal moments of $\{S(t)\}$. 
These \kwo{infinitesimal moments} and their ratio (the dispersion index $D_{dS}$) can be derived by noticing that $\{S(t)\}$ is a compound Poisson process, so that
\bea
\label{eqn:inf-mean}
\mu_{dS}(s) = \lim\limits_{\dt \downarrow 0} \frac{ E \big[ S(t + \dt) | S(t) = s\big] }{\dt} &=& \lim\limits_{\dt \downarrow 0} \frac{\alpha \big( 1 - e^{-\alpha} \big)\dt}{\dt} = \alpha \big( 1 - e^{-\alpha} \big)\\
\label{eqn:inf-var}
\sigma^2_{dS}(s) = \lim\limits_{\dt \downarrow 0} \frac{ V \big[ S(t + \dt) | S(t) = s\big] }{\dt} &=& \lim\limits_{\dt \downarrow 0} \frac{\alpha(1+ \alpha ) \big( 1 - e^{-\alpha} \big) \dt}{\dt} = \alpha(1+ \alpha ) \big( 1 - e^{-\alpha} \big)\\
\label{eqn:inf-disp}
D_{dS}(s) = \frac{\sigma^2_{dS}(s)}{\mu_{dS}(s)} &=&1 + \alpha
\eea
The \kwo{infinitesimal} quantities~\eqref{eqn:inf-mean}--\eqref{eqn:inf-disp} allow further analysis of the difference in the role of $\alpha$ in $\{S(t)\}$ and in $\{X(t)\}$. 
In addition, \kwo{they} are also important for applied modelers interested in bounding simulation computational load, which is a key element when applying the plug-and-play inference methods pointed to in the introduction. 
Computational load can be bound by approximating $\{S(t)\}$ at certain times using deterministic ordinary differential equations (ODEs), as explored for example in \citet{haseltine2002}. 
For example, the unit-rate Poisson time $\{N(t)\}$ of Corollary~\ref{cor:PP} could be approximated by function $n(t)=t$, giving the approximating ODE $dn(t)/dt = 1$. 
Other \kwb{approximating ODEs} regarding Corollary~\ref{cor:PP} could be $dx(t)/dt = \alpha$ and $ds(t)/dt = \alpha \big( 1 - e^{-\alpha} \big) \neq dx\big( n(t) \big) /dt = \alpha$. 
The \kwb{approximating ODE} $ds(t)/dt = \alpha$ is correct nevertheless if the time change used is instead the \kwp{gamma process} used to derive~\eqref{eqn:gamma-rate} \citep{breto2011}. 

The \kwp{gamma time change} used to derive~\eqref{eqn:gamma-rate} has \kwy{disadvantages} when compared to our \kwo{Poisson time change}, regardless of the \kwdb{base process}. 
A fundamental \kwy{disadvantage} is that our approach to finding the closed-form expressions in Theorem~\ref{thm:rates} does not work, basically because the gamma trajectories are more complex (more details in Note~\ref{note:gamma} in~\ref{app:proof}). 
A minor \kwy{disadvantage} is that it introduces an additional parameter, $\tau$. 
This additional parameter complicates the preceding non-trivial comparison of the role of $\alpha$ in $\{S(t)\}$ with its role in $\{X(t)\}$. 
Of course, $\tau$ makes $\{S(t)\}$ more flexible, e.g., $D_{dS}=1 + \alpha \tau$ \citep{breto2011}, and could always be fixed at some arbitrary value, say $\tau=1$, which would give the same $D_{dS}$ for both the Poisson and the gamma time change. 
In this sense, the \kwo{Poisson choice} makes $\{S(t)\}$ simpler and still infinitesimally over-dispersed. 
The \kwo{Poisson time change} could be replaced by a compound Poisson time change. 
This \kwo{extension} would allow for additional parameters that could play the role of $\tau$, providing additional flexibility. 
\kwo{It} would also still permit proving Theorem~\ref{thm:rates} with some minor modifications (which are beyond the scope of this paper). 
The above analysis comparing the role of parameters before and after changing time and with different time changes is also possible for \kwdb{base processes other than Poisson}. 

\kwdb{Base processes} for which the distribution of $X(1)$ is well-known also include \kwy{linear birth and linear death processes} \citep[see for example ][]{bharucha1960}, for which we next provide closed-form rates and \kwo{moments} before considering applying Theorem~\ref{thm:rates} to processes for which this distribution is \kwp{more difficult to obtain}. 
\begin{corollary}[\kwy{\textbf{Negative-binomial Poisson process}}]\label{cor:NBP}
Let $ \lambda_X(x) = \beta x \Ind{x > 0}$, i.e., let $\{X(t)\}$ be a linear birth process with individual birth rate $\beta \in \R_{>0}$. 
The transition rates of $\{S(t)\}$ are the probability mass function of a negative binomial random variable parameterized by success probability $1 - e^{-\beta}$ and number of failures until stopping $s$, i.e., $\q[s,k][S][] = {s + k - 1 \choose k} \left( e^{-\beta} \right)^s \left(1 - e^{-\beta} \right)^k$ for $s \in \N_{>0}$ and its rate function is $\lambda_{S}(s) = 1-e^{-\beta s \Ind{s > 0}}$.
\end{corollary}
\begin{corollary}[\kwy{\textbf{Binomial Poisson process}}]\label{cor:BP}
Let $ \lambda_X(x) = \delta (d_0 - x) \Ind{x < d_0}$, i.e., let $\{X(t)\}$ be the counting process associated with a linear death process with individual death rate $\delta \in \R_{>0}$ and initial population size $d_0 \in \N_{>0}$. 
The transition rates of $\{S(t)\}$ are the probability mass function of a binomial random variable parameterized by success probability $1 - e^{-\delta}$ and number of trials $(d_0 - s)$, i.e., $\q[s,k][S][] = { d_0 - s \choose k} \left( 1 - e^{-\delta} \right)^k \left( e^{-\delta}\right)^{ \left( d_0 - s \right) - k}$ with $k = 1, \ldots, d_0 - s\;\,$for each $s = 0, \ldots, d_0-1$ and its rate function is $\lambda_{S}(s) = 1-e^{-\delta (d_0 - s) \Ind{s < d_0} }$.
\end{corollary}
In both Corollary~\ref{cor:NBP} and~\ref{cor:BP}, the \kwo{event size moments} are the moments of the corresponding negative binomial and binomial distributions. 
\kwo{Infinitesimal moments} are provided below for the linear death process of Corollary~\ref{cor:BP} only. 
(Providing them for Corollary~\ref{cor:NBP} would require a justification to pass the $\dt$-limit inside the expectation, a technicality that at this point we prefer leaving for future research.) 
If $\{S(t)\}$ is the process from Corollary~\ref{cor:BP}, then 
\bea
\nonumber
\mu_{dS}(s) &=& \sum_{k=1}^{d_0 - s}{ k \; \lim_{\dt \downarrow 0} \dt^{-1} P \Big ( S(t+\dt) = s + k \; \big | \; S(t) = s \Big ) } = \sum_{k=1}^{d_0 - s}{ k \; P \Big ( X(1) = s + k \; \big | \; X(0) = s \Big ) } = (d_0 - s) \Big (1- e^{-\delta} \Big )\\
\nonumber 
\sigma^2_{dS}(s) &=& \sum_{k=1}^{d_0 - s}{ k^2 \; P \Big ( X(1) = s + k \; \big | \; X(0) = s \Big ) } + \lim_{\dt \downarrow 0} \frac{\mu_{dS}^2(s) \dt^2 + o(\dt)}{\dt}= \mu_{dS}(s) \; \Bigg[ 1 + (d_0 - s - 1) \Big( 1 - e^{-\delta} \Big ) \Bigg] \\
\nonumber 
D_{dS}(s) &=& 1 + (d_0 - s - 1) \Big( 1 - e^{-\delta} \Big )
\eea
These \kwo{infinitesimal moments} are again different from those of a binomial gamma process, but in both cases $\{S(t)\}$ is infinitesimally equi-dispersed for $s=d_0-1$ \citep{breto2011}. 
Regarding \kwb{approximating ODEs} corresponding to Corollary~\ref{cor:BP}, they are similar to those of Corollary~\ref{cor:PP}: $dx(t)/dt = \delta (d_0 - x) \Ind{x < d_0}$ and $ds(t)/dt =  (d_0 - s) \Big (1- e^{-\delta}\Big )\Ind{s < d_0}$. 
All of the preceding analysis of the role of parameters has been greatly facilitated by \kwp{nice and tractable distributions of $X(1)$}. 

When the \kwp{distribution of $X(1)$ is difficult to obtain}, the closed-form results provided by Theorem~\ref{thm:rates} inherit this \kwy{lack of tractability}, which, in the most extreme case, can force the use of \kwo{approximations} in both univariate and \kwdb{multivariate settings}. 
An example of a \kwy{distribution that is difficult to handle} is that of non-linear death processes with $\lambda_X(x) = x (d_0 - x)\Ind{x < d_0}$, which is of current interest in the context of bacterial disinfection \citep{chou2005} and for which Theorem~3 in \citet{billard1979} gives intricate closed-form expressions for $P \big( X(1) = s + k \; | \; X(0) = s \big)$. 
In this case, results can be obtained without the need of any approximation. 
\kwo{Approximation} of the distribution of $X(1)$ is an option if this distribution is unavailable in closed form. 
Such \kwo{approximation} can be done numerically but numerical approximations can be avoided at least for death processes with general death rate, including the following non-linear death processes: logistic and exponential death processes \citep[which are of interest to model invasion of larvae,][]{faddy1999}; Bass death processes \citep[to model innovation diffusion in social groups, e.g.,][]{karmeshu2007}; and power death process \citep[to model chemical formation of activated carbons,][]{fan2011}. 
For such processes, it is possible to truncate the infinite sum for $P \big( X(1) = s + k \; | \; X(0) = s \big)$ provided in Theorem~2 in \citet{billard1980}. 
These numerical or truncation \kwo{approximations} should in any case be expected to give more reliable results than empirically attempting to approximate the limit and expectation in~\eqref{eqn:analytic} to obtain approximate time-changed transition rates. 
The \kwo{transition rates} considered so far refer to univariate counting processes only but can also be used to construct \kwdb{multivariate Markov counting systems}. 

\kwdb{Over-dispersed Markov counting systems} have been constructed combining the transition rates of Poisson gamma and binomial gamma processes as \kwy{building blocks} of a multivariate counting system in the context of epidemiological susceptible-infectious-recovered (SIR) \kwo{compartmental models} \citep{breto2011, breto2012}. 
These \kwy{blocks} based on gamma time changes can be replaced by the Poisson-based blocks presented in this paper, resulting in alternative SIR compartmental models. 
Such alternative SIR \kwo{models} can then be considered as a more parsimonious alternative to the gamma-based models to test for over-dispersion (i.e., for model misspecification) and to improve the fit of the base model to data. 
%
%
¥
\section*{Acknowledgements}
This work was supported by Spanish Government Project ECO2012-32401 and Spanish Program \emph{Juan de la Cierva} (JCI-2010-06898). 
%
\bibliographystyle{elsarticle-harv}
%
%
\bibliography{references}
%
%
%
%
%
\appendix
\renewcommand{\thesection}{Appendix \Alph{section}}
\renewcommand{\theequation}{\Alph{section}.\arabic{equation}}
\setcounter{equation}{0}
\setcounter{section}{0}
%
%
¥
\section{}
\label{app:proof}
Before proving Theorem~\ref{thm:rates}, we give a formal definition of time change. 
Time change can be approached rigorously by defining an underlying probability space and by paying attention to the measure theory involved \citep[as in for example][]{sato1999, barndorff2010}. 
Such rigor is not needed for our proof but beginning with a probability space $(\Omega, \mathcal{F}, P)$ clarifies it. 
On this \kwy{probability space}, the collection (indexed by $t \in \R_{\geq 0}$) of random variables $X_t(\omega): \Omega \rightarrow \N_{\geq 0}$ defines the \kwo{base counting process} $\{X(t)\}$ of Theorem~\ref{thm:rates}. 
If instead of indexing by $t$, we index by $\omega$ (which we stress by switching to lowercase), $x_{\omega}(t)$ is an $\N_{\geq 0}$-valued deterministic function called the process trajectory or sample path. 
Composing trajectory $x_{\omega}(t)$ with that of Poisson process $\{N(t)\}$ is equivalent to defining the time-changed process $\{S(t)\}$ as the collection of random variables such that, for all $t$,
\begin{eqnarray}
\label{eqn:subordination}
S_t(\omega) = X_{N_t (\omega)} (\omega)
\end{eqnarray}
(as in \citealp{sato1999}) with trajectories $s_{\omega}(t) = x_{\omega} \big( n_{\omega}(t) \big)$. 
%
%
%
%
\begin{proof}[ \bf{Proof of Theorem~\ref{thm:rates}} ]
The result is proved by inferring, from the \kwdb{properties of the composed trajectory}, the \kwp{conditional distribution of inter-event times} in $\{S(t)\}$, which determines the rate function in~\eqref{eqn:time-changed-rate}, and \kwdb{the conditional distribution of event sizes} in $\{S(t)\}$, which determines the transition rates in~\eqref{eqn:time-changed-rates}.

A key \kwdb{property of the trajectories} of the processes involved is that they are flat with step or jump \kwy{discontinuities}, which carry information about the random \kwp{inter-event times}. 
Let these \kwy{discontinuity points} of trajectories $x_{\omega}(t)$, $n_{\omega}(t)$ and $s_{\omega}(t)$ be denoted by $\{\delta_1^x, \delta_2^x, \ldots \}$, $\{\delta_1^n, \delta_2^n, \ldots \}$ and $\{\delta_1^s, \delta_2^s, \ldots \}$ respectively.
These three \kwy{collections} are related, as can be seen from trajectory $s_{\omega}(t)$ implied in~\eqref{eqn:subord-traj} by definition~\eqref{eqn:subordination}: 
\begin{eqnarray}
\label{eqn:subord-traj}
s_{\omega}(t) \equiv x_{\omega}(n_{\omega}(t)) = \left\{ 
  \begin{array}{l l}
    x_{\omega}(0) & \quad \text{for $0 \leq t < \delta_1^n$}\\
    x_{\omega}(1) & \quad \text{for $\delta_1^n \leq t < \delta_2^n$}\\ 
    x_{\omega}(2) & \quad \text{for $\delta_2^n \leq t < \delta_3^n$}\\ 
    \;\;\;\;\vdots 		  & \quad \quad \quad \quad  \vdots
  \end{array} \right.
\end{eqnarray}

From~\eqref{eqn:subord-traj}, \kwy{$\big \{\delta_1^s, \delta_2^s, \ldots \big \}$} will be the subsequence $\big\{\delta_{g_1}^n, \delta_{g_2}^n, \ldots \big\}$ of $\big\{\delta_1^n, \delta_2^n, \ldots \big\}$ that skips $\delta_i^n$ if and only if $x_{\omega}(i) = x_{\omega}(i-1)$, e.g., $g_1=1$ (or equivalently $\delta_1^s = \delta_1^n$) if and only if $x_{\omega}(1) \neq x_{\omega}(0)$, otherwise $\delta_1^n$ will be skipped and $\delta_1^s$ will be some later discontinuity point of $n_{\omega}(t)$.
\kwy{Points \kwy{$\big \{\delta_1^s, \delta_2^s, \ldots \big \}$}} are in fact a set of realizations of the random event times in $\{S(t)\}$, providing a means to characterizing the \kwp{distribution of the $i^{th}$ inter-event time} of $\{S(t)\}$.
This random inter-event time is denoted by $T_{i}^S$ and its realizations by $t_{i}^S$. 
(In addition, we artificially set the first event time to $0$, so that $T_1^S$ coincides with the first non-zero event time.) 

The conditional \kwp{distribution of inter-event times} $T_{i}^S$ is easiest described via an example of trajectories, which later facilitates describing \kwdb{the conditional distribution of event sizes}. 
Consider for example the trajectories in Figure~\ref{fig:trajectories} below: the realized first inter-event time $t_1^S = t_1^N + t_2^N + t_3^N$. 
In general, $T_1^S = T_1^N + \ldots + T_{G_1}^N$, where $G_1$ follows a geometric distribution with succes probability $\pi ( X(0) ) = 1 - e^{\lambda_X( X(0) )}$. 
Since $X(0) = S(0)$ by definition of $N(0)$, conditionally on $S(0)=s$, $T_1^S \sim \text{exponential}(\pi (s) )$. 
Returning to Figure~\ref{fig:trajectories}, $t_2^S = t_4^N + t_5^N$. 
In general, $T_2^S = T_{G_1}^N + \ldots + T_{G_1 + G_2}^N$, where $G_2$ follows again a geometric distribution now with success probability $\pi ( X(G_1) )$. 
Since $X(G_1) = S(T_1^S)$, conditionally on $S(T_{1}^S) = s$, $T_2^S \sim \text{exponential}(\pi (s) )$. 
It follows that in general, conditionally on $S \big( \sum_{i=1}^{i-1}T_{i}^S \big) = s$, $T_i^S \sim \text{exponential}(\pi (s) )$, which shows~\eqref{eqn:time-changed-rate}. 

Regarding jump or \kwdb{event sizes}, define the first jump size $J_1 \equiv S(T_1^S) - S(0)$. 
In Figure~\ref{fig:trajectories}, $J_1 = X(3) - X(0) = X(3) - X(2)$. 
In general, $J_1 = X(G_1) - X(G_1 - 1)$. 
Since $X(G_1-1) = X(0) = S(0)$, conditionally on $S(0)=s$, the probability of a first jump of size $k>0$ is 
\bea
\nonumber
P \Big( J_1 = k \; \big | \; S(0) = s \Big) &=& \frac{P \Big( X(G_1) = s + k \; \big | \; X(G_1 -1 ) = s \Big)}{\pi(s)}\\ 
\nonumber
&=& \frac{P \Big( X(1) = s + k \; \big| \; X(0) = s \Big)}{\pi(s)}
\eea 
where the last equality follows by time homogeneity of $\{X(t)\}$. 
This result generalizes to the $i^{th}$ jump $J_i$ by an argument analogous to the preceding one (used to generalize the result for $T_1^S$ to $T_i^S$), which shows~\eqref{eqn:time-changed-rates}.
\end{proof}
\begin{note}[\textbf{On extending the proof to other time changes}]
\label{note:gamma}
The distributions of $T_i^S$ and $J_i^S$ are not as straightforward to derive using the strategy used in this proof if $\{N(t)\}$ is not a Poisson process. 
For example, if it is a gamma process, $T_i^S$ is still exponential but the distribution of $J_i^S$ depends on that of the over-shoot of the hitting time of $\{N(t)\}$ to the corresponding $\delta^X_i$. 
However, this over-shoot distribution seems to have remained elusive and we are only aware of approximations to its moments \citep[see for example][]{nicolai2009}. 
If instead of a gamma process, $\{N(t)\}$ is some counting process (of finite activity) other than Poisson, the strategy followed in this proof remains effective, although the proof becomes more involved. 
\end{note}
\begin{figure}[H]
\centering
	\includegraphics[scale=0.35]{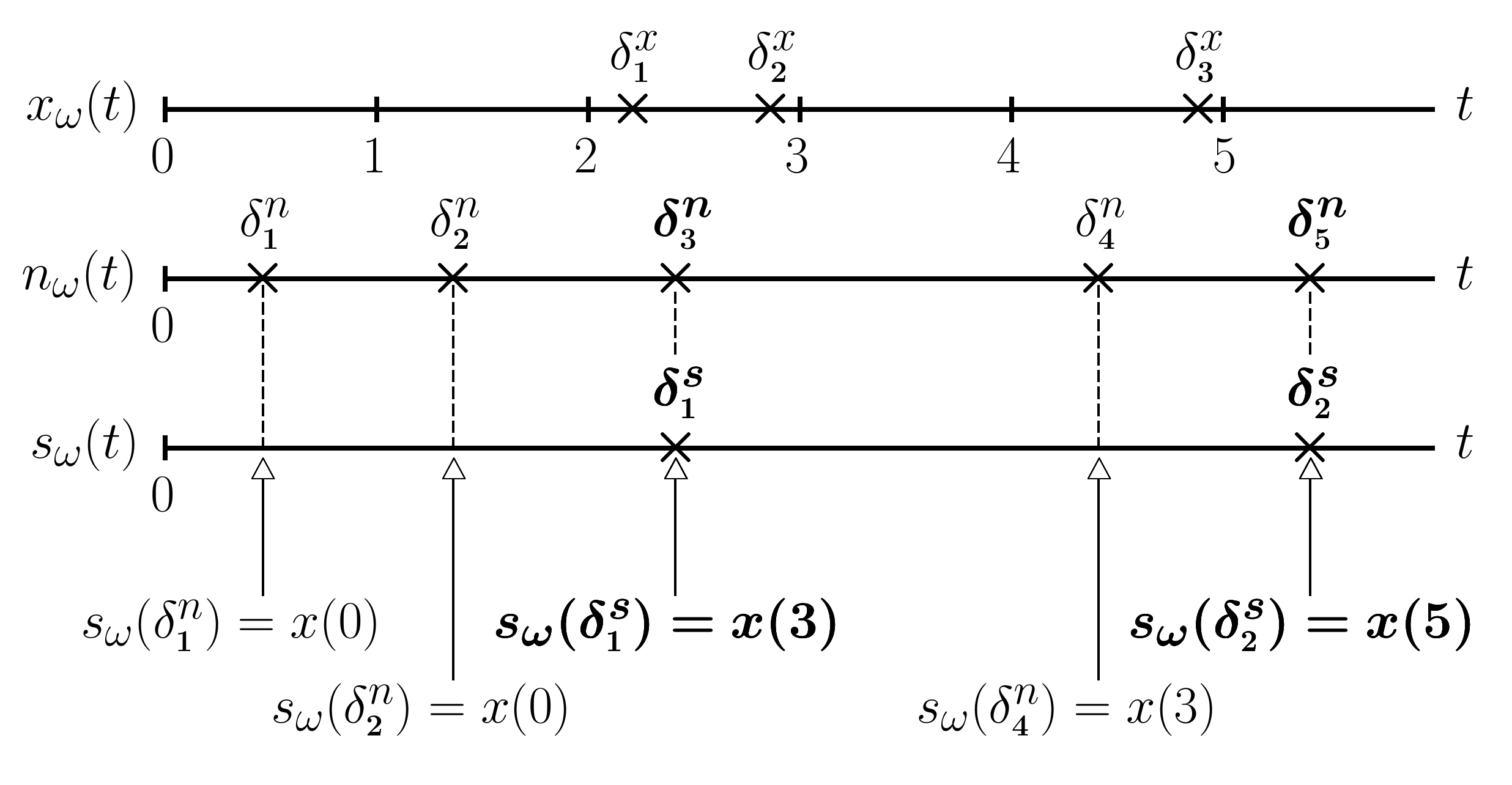}
\caption{Example of composition of trajectories. 
The three horizontal lines represent time and show (marked with $\bm{\times}$) realizations of several event times (i.e., trajectory discontinuities) for each process. 
Since the first event in $\{X(t)\}$ does not occur until time interval $(2,3]$, event time realizations $\delta_1^n$ and $\delta_2^n$ have no effect on $\{S(t)\}$ and $\delta_3^n$ triggers the first event time realization of $\{S(t)\}$ (i.e., $\delta_1^s = \delta_3^n$), turning the two events of size one in $\{X(t)\}$ occurred at $\delta_1^x, \delta_2^x \in (2,3]$ into a simultaneous event of size two in $\{S(t)\}$. 
After $\delta_1^S$, there is $\delta_2^S = \delta_5^N$, at which the single event (again of size one) occurred at $\delta_3^X \in (4,5]$ is shifted to $\delta_2^S$ and remains of size one. 
\label{fig:trajectories}}
\end{figure}
\end{document}